\documentclass[reqno]{amsart}
\usepackage{epsfig,latexsym,amsfonts,amssymb,amsmath,amscd}
\usepackage{amsthm}
\usepackage{tikz-cd}
\usepackage[mathscr]{eucal}
\usepackage{mathrsfs}
\usepackage{mathbbol}
\usepackage{array}
\usetikzlibrary{matrix,arrows,decorations.pathmorphing}
\usepackage{oldgerm,units,color}

\usepackage{pst-node}
\usepackage{graphicx}

\usepackage{eepic, epic}

\usepackage{ifthen}

\def\mtw{\cdot_{\operatorname{tw}}}
\def\tw{_{\operatorname{tw}}}
\def\Hy{\mathcal H}
\def\hyzero{\zero_{\mathcal H}}

\def\det{\operatorname{det}}
\def\GKdim{\operatorname{GK-dim}}

\def\tTA{\tT_{\mcA}}
\def\tTM{\tT_{\mcM}}

\def\tTz{ \tT_\operatorname{\zero}}

\newcommand{\nsets}{{\mathcal P}^*}

\def\module0{module$^\dagger$}

\def\ker{\operatorname{ker}}

\def\sgn{{\operatorname{sgn}}}

\def\ssemiring0{$s$-semiring$^\dagger$}
\newtheorem*{nothma}{\textbf{Theorem A}}
\newtheorem*{nothmb}{\textbf{Theorem B}}
\newtheorem*{nothmc}{\textbf{Proposition C}}
\newtheorem*{nothmd}{\textbf{Proposition D}}
\newtheorem*{nothme}{\textbf{Proposition E}}

\usepackage{tikz}
\usepackage{epsfig,latexsym,amsfonts,amssymb,amsmath,amscd,graphics,epic}
\usepackage{amsfonts,amssymb,amsmath,amscd,amsthm}
\usepackage[mathscr]{eucal}
\usepackage{mathrsfs}

\usepackage{mathbbol}

\usepackage{oldgerm,units}
\usepackage{wrapfig,epsfig}

\usepackage{ifthen}
\usepackage{mathbbol}

\usepackage{amsthm}
\usepackage[mathscr]{eucal}
\usepackage{mathrsfs}
\usepackage{mathbbol}
\usepackage{oldgerm,units}
\usepackage{wrapfig}

\newtheorem{theorem}{Theorem}[section]

\newtheorem{definition}[theorem]{Definition}

\newtheorem{example}[theorem]{Example}
\newtheorem{remark}[theorem]{Remark}

\newcommand{\Real}{\mathbb R}

\newcommand{\Net}{\mathbb N}

\newcommand{\one}{\mathbb{1}}
\newcommand{\zero}{\mathbb{0}}

\newcommand{\trop}[1]{\mathcal{#1}}

\newcommand{\tG}{\trop{G}}

\newcommand{\tT}{\trop{T}}

    \ifx\proof\undefined
    \newenvironment{proof}{
    \smallskip
    \noindent\emph{Proof.}}{\hfill\(\Box\)
    \bigskip
    } \fi

\newcommand{\ifdef}[3]{\ifthenelse{\equal{#1}{true}}{#2}{#3}}

\definecolor{lgray}{gray}{0.90}

\def\mcA{\mathcal A}
\def\mcW{\mathcal W}

\def\mcK{\mathcal K}

\def\mcI{\mathcal I}

\def\({\left(}
\def\){\right)}

\def\Z{{\mathbb Z}}

\def\pipe{{\underset{{\ \, }}{\mid}}}

\def\vsemifield0{$\nu$-semifield$^\dagger$}
\def\vsemiring0{$\nu$-semiring$^\dagger$}

\def\pipe1{{\underset{{1}}{\mid}}}
\def\lmod1{\mathrel  \pipe1  \joinrel \joinrel =}

\def\CFunFF1{\operatorname{CFun} (F,F)}

\def\semiring0{semiring$^{\dagger}$}
\def\Semiring0{Semiring$^{\dagger}$}
\def\Semirings0{Semirings$^{\dagger}$}
\def\semidomain0{semidomain$^{\dagger}$}
\def\semifield0{semifield$^{\dagger}$}
\def\semifields0{semifields$^{\dagger}$}
\def\vsemifields0{$\nu$-semifields$^{\dagger}$}
\def\domain0{domain$^{\dagger}$}
\def\predomain0{pre-domain$^{\dagger}$}
\def\predomains0{pre-domains$^{\dagger}$}
\def\domains0{domains$^{\dagger}$}

\def\vdomains0{$\nu$-domains$^{\dagger}$}

\def\domains0{domains$^\dagger$}

\def\mcH{\mathcal H}
\def\mcM{\mathcal M}
\def\mcK{\mathcal K}
\def\mcN{\mathcal N}

\newcommand{\etype}[1]{\renewcommand{\labelenumi}{(#1{enumi})}}
\def\eroman{\etype{\roman}}

\def\pipe{{\underset{{\tG}}{\mid}}}

\def\lmod{\mathrel  \pipe \joinrel \joinrel =}
\def\pipe{{\underset{{\tG}}{\mid}}}

\newtheorem{thm}[theorem]{Theorem}

\newtheorem*{thm*}{Theorem}
\newtheorem{lem}[theorem]{Lemma}
\newtheorem{rem}[theorem]{Remark}
\newtheorem{prop*}{Proposition}
\newtheorem{conj*}{Conjecture}

\newtheorem{prop}[theorem]{Proposition}
\newtheorem{defn}[theorem]{Definition}
\newtheorem*{examp*}{Example}
\newtheorem*{examples*}{Examples}
\newtheorem*{remark*}{Remark}
\newtheorem*{defn*}{Definition}
\newtheorem*{note*}{Note}

\newtheorem{ques}[theorem]{Question}

\def\R{\Real}
\def\Rmax{\R_{\operatorname{max}}}
\def\la{\lambda}
\def\La{\Lambda}
\def\tT{\mathcal T}

\def\tTz{\tT_\zero}

\def\tTM{\tT_{\mathcal M}}

\numberwithin{equation}{section}

\def\M0{M_{\zero}}

\def\pMN{\phi_{\mcM,\mcN}}
\def\pMK{\phi_{\mcK,\mcM}}
\def\pAA{\phi_{\mcA,\mathcal{A}_0}}

\def\PS{P}

\def\Cong{\Phi}

\def\Diag{{\operatorname{Diag}}}

\def\semirings0{semirings$^\dagger$}

\newcommand{\nPS}[1]{\PS_{(!#1)}}
\newcommand{\nPSo}[1]{\nPS{\one}}

\newcommand{\adj}[1]{\operatorname{adj}({#1})}

\begin{document}


\title[$\tT$-semiring pairs ]
{$\tT$-semiring pairs }

%
\author[J.~Jun]{Jaiung~Jun}
\address{Department of Mathematics, State University of New York at New Paltz, New Paltz, NY 12561, USA} \email{jujun0915@gmail.com}
\author[K.~Mincheva]{Kalina~Mincheva}
\address{Department of Mathematics, Tulane University, New Orleans, LA 70118, USA}
\email{kmincheva@tulane.edu}
\author[L.~Rowen]{Louis Rowen}
\address{Department of Mathematics, Bar-Ilan University, Ramat-Gan 52900, Israel} \email{rowen@math.biu.ac.il}

\makeatletter
\@namedef{subjclassname@2020}{%
    \textup{2020} Mathematics Subject Classification}
\makeatother




%



\begin{abstract} 
{We develop a general axiomatic theory of algebraic pairs, which simultaneously generalizes several algebraic structures, in order to bypass negation as much as feasible. We investigate several classical theorems and notions in this setting including fractions, integral extensions, and Hilbert's Nullstellensatz. Finally, we study a notion of growth in this context.}
\end{abstract}

\maketitle

\tableofcontents




\section{Introduction}

Over recent years an effort has been made to understand tropical
mathematics in terms of various algebraic structures. A major role
has been taken by the max-plus algebra $\Rmax$, but with the drawback that
its structure theory is quite limited by the absence of negation.
Gaubert~\cite{Ga} addressed this issue in his dissertation, followed by \cite{AGG}, studying a
construction which we call ``symmetrization'' in \cite{AGR,Row21}. 
Izhakian~\cite{zur05TropicalAlgebra} introduced a modification, later called ``supertropical algebra'' and studied extensively by Izhakian and Rowen \cite{IR,IRMatrices},
and later with Knebusch. 
Meanwhile an extensive literature developed around hyperstructures \cite {BB1,CC, GJL, Ju,  Vi} and fuzzy rings \cite{Dr, DW}.
These constructions were unified by Lorscheid in ``blueprints''
\cite{Lor1, Lor2}, and carried further into a  ``systemic'' algebraic approach  
taken by \cite{Row21}, and developed in \cite{AGR, GaR,JMR,JuR1}
in order to unify classical algebra with
 the algebraic theories of supertropical algebra, symmetrized semirings, hyperfields, and fuzzy
 rings,
 with some success especially in obtaining theorems about
 matrices, polynomials, and linear algebra. 
 
 The idea of systems in brief is to
 consider a formal ``negation map'' $(-)$ on $\mcA$ satisfying all properties of negation
 except  $a-a=0.$ The zero element no longer plays a role, but is
 replaced by the set of \textbf{quasi-zeros} $\mcA^\circ = \{
b+((-)b): b\in \mcA\}.$  In supertropical algebra the ``negation map''
 actually is the identity, and in a semiring where a negation map is lacking, it can
be provided in the symmetrization process.
Another innovation was the ``surpassing relation'' $\preceq$,    to extend equality in most results.

Although the negation map is very useful, this paper addresses the question as to how much semiring theory can be performed only with the ``surpassing relation'' given below in
Definition~\ref{precedeq07}, which is needed for our version of equations. In addition to the philosophical question of the minimal axiomatic framework needed to carry out the theory, we are motivated by new structures generalizing hyperfields which come up naturally in the study of semirings, given in \S\ref{shyp}.

Many of the relevant notions of systems
are formulated without a negation map on $\mcA$,  by elevating $\mcA^\circ$ to the principal structural role, via
the embedding $\mcA^\circ\to \mcA$.
 Then we can formulate relative versions
of algebraic concepts which could be applied to tropical (and other) situations.
In other words, we continue the approach of systems, but with a  different emphasis
which may provide extra intuition, in an effort to obtain the greatest generality in which the algebraic structure theorems of the Artin-Krull-Noether theory are available.
One such direction of the structure theory taken in \cite[\S3 and \S4]{JuR1} was the spectrum of prime congruences.

 Our main focus, extending \cite[\S3 and \S4]{JuR1}, is on the structure of \textbf{pairs} $(\mcA,\mcA_0)$ acted upon by a set $\tT.$
 See Definition~\ref{symsyst} for the formal set-up. In brief, $(\mcA,\mcA_0)$ is a \textbf{$\tT$-semiring pair} when $\mcA_0 
 \subseteq \mcA$ are semirings with compatible $\tT$-actions, and $\tT$ is a multiplicative submonoid of $\mcA$, where $\mcA_0\cap \tT =\emptyset$ and $\tT$ spans $(\mcA,+)$. The   $\tT$-pair $(\mcA,\mcA_0)$ is \textbf{shallow} if $\mcA = \tT \cup \mcA_0. $
 
Likewise, for modules (called ``semimodules'' in the semiring literature), we suppress the negation map, and simply
consider pairs $(\mcM,\mcN)$ of modules over $(\mathcal A, \mcA_0)$,  together with a map $\pMN: \mcN\to \mcM$   where in applications  $\pMN(\mcN) \subseteq \mcM$
could be identified with the quasi-zeroes. 
Ideals in algebra are replaced by congruences with the ``twist product''.
We go through the basic structure theory, starting with ``prime'' and ``semiprime''defined by means of the twist product. Then we compare notions of algebraicity,  and compute the growth of some basic pairs. Throughout we search for the
 precise hypotheses necessary for the various theorems.

One also can  manage
 without negation maps when studying
  polynomials (\S\ref{pol}), as well as pairs of  fractions. The Ore condition for semigroups is
  well known, cf.~\cite{Lj}.
Following \cite[p. 349]{Row06} as a model, we also introduce the notion of Ore condition for pairs, under which one obtains the pair of fractions of a $\tT$-semiring pair $(\mcA,\mcA_0)$ (Theorem~\ref{OreP}).

 One goal of any algebraic theory is to obtain some analog of Hilbert's Nullstellensatz, which says that every radical ideal of the affine polynomial algebra is a set of zeroes of polynomials. This is tricky in the semiring set-up because the structure is described in terms of congruences, not ideals, but we propose an approach in Appendix A.
 
\subsection{Main Results}

\begin{nothma}$($Theorem~\ref{semip}$)$
 A congruence $\Cong$ on a $\tT$-semiring pair $(\mcA,\mcA_0)$ is  semiprime if and only if it is the intersection of a nonempty set of prime congruences.
\end{nothma}

See Definition~\ref{free0} for ``$(\preceq)$-base.''

\begin{nothmb}$($Theorem~\ref{OreP}$)$ 
Given a $\tT$-pair  $ (\mcA,\mcA_0)$  with $\tT$ satisfying the     Ore condition  with respect to $S$,  one can define an equivalence relation on $\mcA\times S,$ by $(b_1,s_1) \equiv (b_2,s_2)$ iff there are $c,c'\in \mcA$ for which $c b_1=b_2c'$.
 
Write $S^{-1}\mcA$ for  $\{s^{-1}b:= (b,s) : b\in \mcA,\, s\in S\},$  $S^{-1}\mcA_0$ for  $\{s^{-1}b:= (b,s) : b\in \mcA_0,\, s\in S\},$ and $S^{-1}\tT$ for  $\{s^{-1}a:= (a,s) : a\in \tT,\, s\in S\}.$ Then
$(S^{-1}\mcA,S^{-1}\mcA_0)$ is an $S^{-1}\tT$-pair.
\end{nothmb}

\begin{nothmc}
$($\textbf{Artin-Tate for pairs}, Theorem~\ref{AT}$)$
 Suppose $(\mcW\mcW_0)$ is an affine 
 semialgebra pair over $\mcA$
 and has a finite $(\preceq)$-base $B$ over a central semialgebra $\mcK \subset \mcW$. Then $\mcK$ is affine over $\mcA.$ 
\end{nothmc}
We need some more preparation for the next result.

A pair $(\mcA, \mcA_0)$ is \textbf{$\preceq$-nondegenerate} if $f(\tT) \not \subseteq \mcA_0$ for any tangible polynomial $f$.

\begin{nothmd}$($Theorem~\ref{sN}$)$
  Suppose $(\mcA, \mcA_0)$ is   a $\preceq_0$ -nondegenerate, shallow semiring pair, and
  \  $(\mcW, \mcW_0)$  is a centralizing extension of $(\mcA,\mcA_0)$, with $y\in \tT_\mcW$ such that $(\mcA[y], \mcA_0[y])$ is tangibly separating.\footnote{(See Example~\ref{Making}(iv) below).} Let 
  $\tT' = \{ a y^i : a \in \tT,  i \in \Net\}.$
  Let $(\mcK, \mcK_0)$ be  the $\tT'$-semifield of fractions of  $(\mcA[y], \mcA_0[y])$. If  $(\mcK, \mcK_0)$  is affine, then
     $y$ is  congruence algebraic\footnote{(See Definition ~\ref{Making}(ii) below).} over the pair
  $(\mcA, \mcA_0)$.

 \end{nothmd}

The Nullstellensatz does not hold in this setting in general (Example~\ref{nonin}), but we prove that in some circumstances (including the classical case and tropical case), a version of the Nullstellensatz holds. 

\begin{nothmd}$($Theorem~\ref{NullS}$)$
Suppose $(\mcW,\mcW_0)$ is an affine $\tT$-semifield pair over a shallow, $\preceq$-nondegenerate $\tT$-semiring pair $(\mcA,\mcA_0)$, and also having the property that  $\tT$-{algebraic} implies integral. Then $(\mcW,\mcW_0)$ is integral over $(\mcA,\mcA_0)$.
\end{nothmd}

Finally, we study growth of pairs and prove the following. 
   
\begin{nothme}$($Proposition~\ref{Jat}$)$
Any $\tT$-semidomain pair $(\mcA,\mcA_0)$ with subexponential growth has the property that for any $a_1,a_1\in \tT$ there are $b_1,b_2\in \mcA\setminus \mcA_0$ such that $b_1a_1+b_2a_2\in \mcA_0$.  
\end{nothme}

\section{Basic notions}\label{BN}

  See \cite{Row19} for a relatively brief introduction of systems; more
details are given in \cite{JMR}, \cite{JuR1}, and \cite{Row21}.
Throughout the paper, we let $\mathbb{N}$ be the additive monoid of
nonnegative integers. Similarly, we view $\mathbb{Q}$
(resp.~$\mathbb{R}$) as the additive monoid of rational numbers
(resp.~real numbers). $\tT$ will always denote a multiplicative monoid with $\one.$ We say that $\tT$ \textbf{acts} on a set $\mcA$ if there is a binary operation $\tT \times \mcA \to \mcA$ satisfying $$(a_1 a_2) b = a_1(a_2
b),\qquad \forall a_1, a_2 \in \tT,  \ b \in \mcA.$$

It is appropriate to turn to  the context of universal algebra, where
one is given sets, called ``algebraic structures,'' or ``$\Omega$-algebras,'' with various
operations, relations, and identities. The 0-ary operations can be thought of
distinguished elements.  Rather than stating the
definitions formally, we refer the reader to~\cite{Jac}, and give the
main instances:
 \begin{defn}\label{de1}
\begin{enumerate}\eroman
\item
A \textbf{magma} is a set $S$ with a binary operation denoted $(+)$
(addition) or $(\cdot)$ (multiplication). At times we also require a
neutral element, written as $\zero$ or $\one$ respectively.

A \textbf{semigroup} is a magma $S$ whose given binary operation
satisfies the law of associativity.

\item A \textbf{semiring}
(cf.~\cite{Cos},\cite{golan92}) $(\mathcal A, +, \cdot, \zero, \one)$ is an
additive abelian semigroup $(\mathcal A, +, \zero)$ and
multiplicative semigroup $(\mathcal A, \cdot, \one)$ satisfying
$\zero b = b \zero = \zero$ for all $b \in \mathcal A$, as well as
the usual distributive laws. 

The semiring predominantly used in tropical mathematics has been the
max-plus algebra, where  $\oplus$ designates $\max$, and $\otimes $
designates $+$. However,
 we proceed with the familiar algebraic notation of addition and
multiplication in the setting under consideration.

\item More generally, a \textbf{bimagma} is a semigroup $(S,+)$ which is also a magma
$(S,\cdot)$. 

\item
A (left) $\tT$-\textbf{module} over a set $\tT$ 
 is a semigroup $( \mathcal A,+,\zero)$,  endowed a $\tT$-action satisfying the
following axioms, for all $a\in \tT$ and $b, b_j \in \mathcal A$:

\begin{enumerate}\eroman
\item
$a \zero=\zero a = \zero.$
\item
$a (\sum _{j=1}^u b_j) = \sum _{j=1}^u (a b_j), \ a \in \mathcal A.$

\end{enumerate}

\item Thus, a \textbf{semialgebra} over
a commutative semiring $\mcA$ is an $\mcA$-module $ \mcM$ which is also a bimagma which satisfies 
$$(ay_1)y_2 = a(y_1 y_2) = y_1 (ay_2)$$
for all $a \in \mcA,$ $y_i \in \mcM.$
An example would be a Lie semialgebra, as defined in~\cite{HH}; also cf.~\cite[Definition~10.4]{Row21} and \cite[Definition~3.4]{CGR}.
\item
A   $\tT$-{module} $\mathcal A $ is \textbf{$\tT$-spanned} if there is a given 1:1 map $\psi: \tT\to \mathcal A $ whose image additively (with $\zero$) generates ~$\mathcal A $. In this situation we identify $\tT$ with $\psi(\tT)$, yielding
 a distinguished set $\tT \subseteq  \mathcal A$. We define $\tT_0 = \tT \cup \{\zero\}$.

\end{enumerate}
\end{defn}

We often suppress the operations and distinguished
elements in the notation. When $\tT$ is ambiguous, we write $\tTA$
to indicate that it is affiliated with $\mcA.$


A \textbf{$\tT$-module homomorphism} $f: \mathcal A_1 \to \mathcal
A_2$ is a function such that $f(ab_1)=\psi(a)f(b_1)$, $f(b_1+b_2)=f(b_1)+f(b_2)$ for
all $a \in \tT$ and $b_i \in \mathcal A_1$, $\psi$ as in Definition~\ref{de1}(iv)(c). Often
we fix an action $f_\tT$ on $\tT$ and consider only those $f$ whose
restriction to $\tT$ is $f_\tT.$


%
 \subsection{Pairs}

 \begin{defn}\label{symsyst} $ $
 \begin{enumerate}\eroman
     \item 
 A \textbf{pair} $(\mcA,\mcA_0)$ is  a pair of   algebraic structures (which should be clear by the context) $\mathcal A$ and $\mathcal A_0$, together with a given homomorphism $\phi_{\mcA,\mcA_0} :\mcA_0\to \mcA.$    
  \item 
 A  $\tT$-\textbf{pair}  is a pair $(\mcA,\mcA_0)$ of  $\tT$-modules over the set $\tT$.  A  $\tT$-pair $(\mcA,\mcA_0)$  is \textbf{admissible} if $\mcA$ is a $\tT$-spanned $\tT$-module and $\phi_{\mcA,\mcA_0}$ is an injection,
  and identifying $\mcA_0$ with the image of $\phi_{\mcA,\mcA_0},$ we have $\mcA_0\cap \tT =\emptyset$. The elements of $\tT$ are called \textbf{tangible}.
 \item  The admissible pair $\tT$-pair $(\mcA,\mcA_0)$ is \textbf{shallow} if $\mcA = \tT \cup \mcA_0. $
 \item 
 A  \textbf{$\tT$-semiring pair}  is an admissible $\tT$-pair $(\mcA,\mcA_0)$ where $\mcA$ is an
(associative) semiring.
A  \textbf{$\tT$-semifield pair} is a $\tT$-semiring pair $(\mcA,\mcA_0)$ where $\tT$ is a group.
\item 
 A $\tT$-pair  $(\mcA,\mcA_0)$ is \textbf{generated} by $S\subseteq\mcA$ if every sub-$\tT$-pair of $\mcA$ containing $S$ and $\phi_{\mcA,\mcA_0}(\mcA_0)$ is $\mcA.$ $(\mcA,\mcA_0)$ is \textbf{finitely generated} if it is generated by a finite set. 
 \end{enumerate}
\end{defn}

We assume throughout this paper that 
$(\mcA,\mcA_0)$ is an admissible $\tT$-semiring pair.

\begin{example}\label{Making}$ $ \begin{enumerate}
 \eroman \item  
(The classical case) $\mcA$ is an algebra and $\mcA_0=\mcI$
 is an ideal of $\mcA$. Then we can take $\phi_{\mcA,\mcA_0}$ to be the identity $\mcA_0\to\mcA$ and consider $(\mcA,\mcA_0)$ as
 $\mcA/ \mcI$. If $\mcI$ is a prime ideal, we could have $\tT = \mcA \setminus \mcI,$ a monoid, and the pair  $(\mcA,\mcA_0)$  is shallow
\item 
(Doubling; analogous to symmetrization in \cite{Row21}). This is a way to create an admissible pair, for 
any additive semigroup $(A,+,0)$. We  
define $\mcA = A\times A,$ $\mcA_0 = \{(a,a): a \in A\}$, $\tT = (A \times 0 )\cup (0 \times A) $ and $\pAA ,$ the identity map.
diagonal. 
 \item 
(The supertropical case) $3a = 2a,$ in the sense that $a+a+a=a+a$ for all $a\in \mcA$, and
 $\mcA_0 =   \{a+a: a\in\tT
 \} =     \{ma: m>1, \ a\in\tT
 \}$, written as $A^\nu$ in the literature.  The pair  $(\mcA,\mcA_0)$ is shallow.
  \end{enumerate}
\end{example}

\subsection{Surpassing relations}

 We  next  provide the pair  with a \textbf{surpassing relation}
$\preceq$ ( \cite[Definition~1.31]{Row21} and also  described in
\cite[Definition~2.10]{JuR1}).

\begin{defn}\label{precedeq07}
A \textbf{surpassing relation} on a $\tT$-pair $(\mathcal A, \mcA_0)$,
denoted
  $\preceq$, is a partial preorder satisfying the following, for elements $b_i \in \mathcal A$:

  \begin{enumerate}
 \eroman
    \item  $\zero \preceq c$
 for any $c\in \mcA_0$.
  \item If $b_1 \preceq b_2$ and $b_1' \preceq b_2'$ for $i= 1,2$ then  $b_1 + b_1' \preceq b_2
   + b_2'.$
    \item   If  $a \in \tT$ and $b_1 \preceq b_2$ then $a b_1 \preceq ab_2.$
    \item      $a \preceq b $ for $a,b \in   \tT$
    implies $a=b.$ (In other words, surpassing restricts to equality on tangible elements.)
   \end{enumerate}
A \textbf{strong surpassing relation} on $\tT$-pair $(\mathcal A, \mcA_0)$ is a surpassing relation   satisfying the following stronger version of $(\mathrm{iv})$: If $b  \preceq a $ for $a \in \tT$ and $b \in \mathcal A$, then $b=a$.
 

\end{defn}

The justification for these definitions is given in
\cite[Remark~1.34]{Row21}. In brief, in proving theorems about  pairs, we often need
equations in tangible elements to be weakened, where $\mcA_0$ takes on the role of ``zero.''


\begin{example} Our main example of a
surpassing relation, denoted $\preceq_0,$ is given by $b_1\preceq b_2$ iff $b_2 = b_1 +y$ for some $y\in \mcA_0.$
\end{example}
This can be defined on any pair, and matches the definition
of systems.

\begin{lem}\label{sh2} If $b\preceq_0 a$ for $a \in \tT$, then  $b\notin \mcA_0$. 

Consequently, if the admissible $\tT$-pair $(\mathcal A, \mcA_0)$ is shallow, then $\preceq_0$ is a strong surpassing relation.
\end{lem}
\begin{proof}
 If $b\in \mcA_0$ then for some $y \in \mcA_0,$ $a = b+ y \in \mcA_0 \cap \tT,$ a contradiction.
 
 Hence, for $(\mathcal A, \mcA_0)$ shallow, $b\in \tT$, so $b=a$.
\end{proof}
\subsection{Negation maps and Property N}

In some cases we can define 
the negation map, the mainstay of \cite{Row21}. A \textbf{negation map}
$(-)$ on $(\mcA,\mcA_0)$ is
 an
additive automorphism of order $\le 2$ satisfying \begin{equation}\label{eq: neg}(-)(bb') =
((-)b)b' = b((-)b'), \quad b+((-)b) \in \mcA_0,  \quad \forall b,b' \in
\mcA,\end{equation}
and $(-)\mcA_0 = \mcA_0.$ When $\mcA$ is a $\tT$-module we also require $(-)$ to be defined on $\tT$, such that $$(-)(ab) = ((-)a)b = a((-)b),\qquad  \forall a\in \tT, \quad b\in \mcA.$$  

We write $b(-)c$ for $b+((-)c)$.
Thus $\mcA_0$ contains the set of
\textbf{quasi-zeros}, denoted as $\mcA^\circ := \{ b(-)b: b\in \mcA\}.$ Often one has $\mcA_0 = \mcA^\circ.$
 
\begin{lem} [{\cite[Lemma~2.11]{JMR}}]\label{nab} If $b_1\preceq_\zero b_2$
in a pair with a negation map, then $b_2 (-) b_1 = b_1 (-) b_2 \succeq
\zero.$   
\end{lem}
\begin{proof}
Write $b_2 = b_1 + c^\circ.$ Then $b_2 (-)b_1 = (b_1+c)^\circ = b_1 (-)b_2 .$
\end{proof}

 A main illustration of a negation map having a different nature is  as follows:

\begin{example}
 Hypersystems $(\mcA, \Hy, (-), \subseteq)$ were defined in \cite{AGR}, where $\mcA_0 = \{
S: \zero \in S\},$ and $\preceq = \subseteq.$
\end{example}

\subsubsection{Weaker versions of negation maps}

Other properties may suffice when we do not have a negation map at our disposal.

\begin{definition}$ $
\begin{enumerate}
 \eroman 

 \item 
 We say that an admissible $\tT$-pair $(\mcA,\mcA_0)$ satisfies
 \textbf{Property N}  if for each $a\in \tT$, there is  $a'\in \tT$ such that $a+a'\in \mcA_0$.
 $(\mcA,\mcA_0)$ is \textbf{neg-compatible} if $a'$ is unique.

 \item A pair satisfying  Property N is \textbf{tangibly separating} if it
satisfies the condition:
 \begin{enumerate}
     \item For each $c\ne a\in \tT$, there is  $a'\in \tT$ such that $c+a'\in \tT$ and $a+a'\in \mcA_0$.
\end{enumerate}
 \end{enumerate}
\end{definition}

\begin{remark}
Given a  neg-compatible, admissible $\tT$-pair $(\mcA,\mcA_0)$, one can define a uniquely quasi-negated triple (\cite[Definition~2,13]{Row21}) $(\mcA,\tT,(-))$ by defining $(-)b$ as follows:

Write $b = \sum a_i$ for $a_i \in \tT$, and define $(-)b: = \sum a+i'$, where   $a_i+a_i'\in \mcA_0.$

In the other direction, given a triple $(\mcA,\tT,(-))$ one defines $ \mcA_0=\mcA^{\circ} $
to get a systemic $\tT$-pair.

Thus, in a sense, neg-compatible $\tT$-pairs only redefine triples 
from \cite{Row21} and from \cite{JuR1}. The set-up here is more
general, and our main objective  is to recast the theory
of systems to see when
negation maps (and ``systemic'') can be omitted, and only to utilize the surpassing relation.
 \end{remark}

Here are some versions of classical structures which need 
not even satisfy Property N, but
at times one can obtain a negation map on a $\tT$-submodule $S$ of $\mcA$, where \eqref{eq: neg} holds on $S$, and $(-)(S\cap \mcA_0) = (S\cap \mcA_0).$ In this case we call $(-)$ a \textbf{partial  negation map} (on $S$). We give some examples motivated by \cite{CGR}.

\begin{example} $ $
\begin{enumerate}\eroman
    \item An \textbf{exterior pair}  is  an admissible $\tT$-pair  $(\mcA,\mcA_0)$  with
    $x^2 \in \mcA_0,$
$xy+yx \in \mcA_0$, for each $x,y\in \mcA.$

 In \cite[\S 4]{CGR}, a partial  negation map is defined
     on the tensors of degree $\ge 2$ in the tensor semialgebra for the exterior pair, by $(-)v\otimes w = w \otimes v$.

 \item A \textbf{Lie pair} is a   pair  $(\mcA,\mcA_0)$ with   \textbf{Lie multiplication}  $[xy]$ satisfying for all $x,y,z\in \mcA$, $y_0 \in \mcA_0$:
 \begin{itemize}
     \item $[xx]\in \mcA_0$,
     \item $[xy_0]\in \mcA_0$,
   \item $[xy]+[yx] \in  \mcA_0$,
      \item $[[xy]z]=[z[yx]],$
      \item $[[xy]z] \preceq [x[yz]]+[[zx]y].$

 \end{itemize}
  \end{enumerate}
  \end{example}

\subsection{Semi-hyperrings}\label{shyp}

Furthermore we can weaken the requirement of having the hyperring negation.

 \begin{defn}\label{hyp} A \textbf{semi-hypergroup} is a set
$(\Hy,\boxplus)$ where
\begin{enumerate}
   \item \label{hyp1}
$\boxplus$ is a commutative binary operation $\Hy \times \Hy \to
\nsets (\Hy)$ (the set of non-empty subsets of $\Hy$), extended
elementwise to $\nsets (\Hy),$ i.e.,
$$S_1 \boxplus S_2 = \{ a_1 + a_2: \ a_1\in S_1, \ \ a_2\in S_2\},$$
which also is associative in the sense that if we
 define
\[ a \boxplus S = S\boxplus a =\bigcup _{s \in S} \ a \boxplus s,
\]
 then $(a_1
\boxplus a_2) \boxplus a_3 = a_1 \boxplus (a_2\boxplus a_3)$ for all
$a_i$ in $\Hy.$
  \item \label{hyp2} We adjoin an element $\hyzero$ called \textbf{the hyperzero}, which is the neutral element:
    $\hyzero \boxplus a =a$.

  \item \label{hyp3} A \textbf{semi-hyperring} is a semi-hypergroup  $\Hy$ including an absorbing  hyperzero $\hyzero$, with multiplication by $\Hy$ in $\nsets$ distributes over addition.
    \end{enumerate}
\end{defn}
\begin{example}\label{precmain0}$ $
Suppose $\tT$ is a semi-hypergroup $\Hy =(\mcH,\boxplus, \zero)$, and $\mathcal{A}$
the subset of the power set of $\mcH$ additively generated by $\{
\{a\} : a\in \mcH\},$ addition defined elementwise. We define $\preceq$ on~$\mcA$ by putting $S_1\preceq S_2$ iff $S_1\subseteq S_2.$ There are two choices for  $\mcA_0$:
\begin{enumerate}\eroman
    \item   $\mcA_0 = \{
S: \zero \in S\}.$
\item    $\mcA_0 = \{
S: \ | S|\ge 2\}.$
\end{enumerate}
 (i) is the customary definition, following {\cite[Definition~2.16]{JuR1}, \cite[Definition~1.70]{Row21}}, and yielding a system, called a  hypersystem, in which  
$(-)a = -a$. However (ii) has the advantage that the pair $(\mcA, \mcA_0)$ is shallow.
\end{example}

We   have the following extension of an idea of Krasner:

 \begin{prop}[cf.~{\cite{krasner}}]\label{prop-CC}
 Suppose $R$ is a commutative semiring
 having a multiplicative subgroup~$G$.
 Then the set of cosets $\mathcal H:=R/G=\{ [r] = rG: r\in R\}$,
 equipped  with the multivalued addition
$$[ r] \boxplus [{r'}] = \{ [{x+x'}] : x\in rG,\ x' \in
 r'G\},$$
 and the multiplication inherited from $R$,
 is a semi-hyperring. In particular, $\mathcal H$
 is a semi-hyperfield if $R$ is a semifield.\end{prop}
\begin{proof} The proofs of \cite{krasner} do not use negation:
$$(a_1+a_2) ([r] )) =  \{ (a_1+a_2)[x] : x\in rG\} =  \{ [{(a_1+a_2)x}] : x\in rG\} =[(a_1+a_2)[r]];$$
 \begin{equation}\begin{aligned} a ([r] \boxplus [{r'}]))& =  \{ a[{x+x'}] : x\in rG,\ x' \in
 r'G\}\\ & =  \{ [{ax+ax'}] : x\in rG,\ x' \in
 r'G\} =[a [r] \boxplus a[{r'}]);\end{aligned}\end{equation}
associativity also is clear.
\end{proof}

Let us take this one step further.

 \begin{prop}\label{prop-CCh}
 Suppose $\Hy$ is a commutative semi-hyperring
 having a multiplicative subgroup~$G$. For $S \subseteq \mathcal{P}(\Hy),$ define the coset $GS = \{ as: a\in \Hy, \ s\in S\}.$

 \begin{enumerate}\eroman
     \item
 Then the set of cosets $\mathcal H:=\Hy/G=\{ [a] = aG: a\in \Hy\}$,
 equipped  with the multivalued addition
$$[ a] \boxplus [{a'}] = \{ [{x+x'}] : x\in aG,\ x' \in
 a'G\},$$
 and the multiplication inherited from $R$,
 is a semi-hyperring $\mathcal H$, which we denote as $\Hy/_{\operatorname{hyp}}G$. In particular, $\mathcal H$
 is a semi-hyperfield if $\Hy$ is a semi-hyperfield.

 Furthermore, if $G\subseteq \hat{G}$ are subgroups of $\Hy$, then   $$\Hy/_{\operatorname{hyp}}\hat G \cong (\Hy/_{\operatorname{hyp}}G)/_{\operatorname{hyp}}\hat G.$$
 \end{enumerate}
 \end{prop}
 \begin{proof}
 (1) is as in Proposition~\ref{prop-CC}. Clearly  $\mathcal H$ is a monoid, and associativity and distributivity are easy to check.

 For (2), take the map sending $S \mapsto \{ ag:\ a \in S,\, g \in \hat{G}.$
     This is the composition   $$S \mapsto S_G:= \{ ag:\ a \in S,\, g \in \hat{G}\}\  \mapsto  \cup _{\bar g \in \hat{G}/G}\,  S_G \bar g. $$
 \end{proof}

\subsection{Congruences}

Classically, one defines homomorphic images by defining a
\textbf{congruence} on an algebraic structure  $\mcA$ to be an
equivalence relation $\Cong$, which viewed as a set of ordered
pairs, is a subalgebra of $\mcA \times \mcA$ which we require to
be disjoint from $\tT \times \mcA_0$.  
\begin{definition}
A \textbf{congruence} on a pair $(\mcA,\mcA_0)$ is a pair
$(\Cong,\Cong_0)$ of congruences on $\mcA$ and $\mcA_0$ such that we get an induced map $\phi_{\mcA,\mcA_0}: \Cong_0 \to \Cong.$ (In the case of admissible pairs, we consider $\Cong_0\subseteq \Cong.$)
\end{definition}

\begin{remark} Any congruence $(\Cong,\Cong_0)$  on $(\mcA,\mcA_0)$  can be applied 
to produce a pair  $(\overline{\mcA},\overline{\mcA_0})$
where $\overline{\mcA} =\mcA/\Cong$    and $\overline{\mcA_0}=\mcA_0/(\mcA_0\cap \Cong_0),$ and  $\overline \phi_{\mcA,\mcA_0} :\overline{\mcA_0}\to \overline{\mcA}$ is the induced homomorphism.

\end{remark}

The \textbf{congruence kernel} $\ker f$ of a homomorphism $f: (\mcA,\mcA_0) \to (\mcA',\mcA'_0)$ is $$\{(y_1,y_2) \in \mcA \times \mcA: f(y_1) = f(y _2)\}$$ and its restriction to $\mcA_0$;  $\ker f$ is easily seen to be a congruence.
Conversely, every congruence $\Cong$ is the congruence kernel of the natural 
homomorphism $y \mapsto [y],$ where $[y]$ is the equivalence class of $y \in \mcA$ (resp.~$\mcA_0$) under $\Cong$  (resp.~$\Cong_0$). 

Often we delete $\Cong_0$ from the notation, when it
is the restriction to $\mcA_0.$
Certain congruences play a fundamental role.

\begin{definition} $ $
\begin{enumerate}
    \item[(i)] 
$\Diag{\mcA}$ denotes the diagonal congruence $\{(b,b): b\in \mcA\},$ also called the \textbf{trivial} congruence.

 \item[(ii)]  
When $\mcA$ is commutative and associative, for any $a\in \mcA$ the congruence $\Cong_a (\mcA)$ generated by $a$ is $\{(a_1a, a_2 a): a_i\in \mcA\}.$
    In the image of $\mcA$ under this congruence, $a$ is identified with $\zero a =\zero,$ so we have the natural homomorphism sending $a\mapsto \zero.$
\end{enumerate}
\end{definition}

\begin{definition} 
$\mcA\times \mcA$ has the \textbf{switch map}
given by $(a_1,a_2)\mapsto (a_2,a_1),$ and the
\textbf{twist product} given by
\[
(a_1,a_1')\mtw (a_2,a_2')= (a_1a_2 + a_1'a_2',a_1a_2'+a_1'a_2).
\]
\end{definition}

\begin{lem}
Any congruence is closed under the switch map and  the twist product.
\end{lem}
\begin{proof}
The switch map is closed,   since any congruence is symmetric. 

For 
the twist product,
$$(a_1a_2 + a_1'a_2',a_1a_2'+a_1'a_2) = (a_1,a_1)(a_2,a_2')+(a_1',a_1')(a_2',a_2).$$
\end{proof}

We would want to identify $\pAA(\mcA_0)$ as ``zero'' in some factor set
of $\mcA$ with respect to $\pAA(\mcA_0)$, but unfortunately the
congruence generated by  $\{(b, b+\pAA(b_0)):\ b\in \mcA,\, b_0
\in \mcA_0\}$ could be larger than one wants.

\section{Structure of associative pairs}

\begin{note*}
For simplicity, for the remainder of this paper we assume throughout that  $(\mcA,\mcA_0)$ is an  admissible $\tT$-semiring pair. The interaction between $\mcA$, $\mcA_0,$ and $\tT$
is crucial.

\end{note*}

Much of the classic structure theory works for pairs since we can avoid  negation by using the twist product.  

\begin{lem}\label{twass}
The twist product is associative.
\end{lem}
\begin{proof} \begin{equation}\begin{split}
    ((a_1,a_1')&\mtw (a_2,a_2'))\mtw(a_3,a_3') =  (a_1a_2 + a_1'a_2',a_1a_2'+a_1'a_2)\mtw(a_3,a_3') \\ &=( (a_1a_2 + a_1'a_2')a_3 + (a_1a_2'+a_1'a_2)a_3', \\ &= (a_1a_2 + a_1'a_2')a_3+ (a_1a_2'+a_1'a_2)a_3',( a_1a_2 + a_1'a_2')a_3'+ (a_1a_2'+a_1'a_2)a_3)
\\ &= (a_1a_2 + a_1'a_2')a_3+ (a_1a_2'+a_1'a_2)a_3',( a_1a_2 + a_1'a_2')a_3'+ (a_1a_2'+a_1'a_2)a_3)) \\ &= a_1a_2a_3 +a_1'a_2'a_3' +a_1a_2'a_3' + a_1'a_2a_3',  a_1a_2a_3'+ a_1'a_2'a_3'+a_1a_2'a_3+a_1'a_2a_3 \\ &=  (a_1,a_1')\mtw ((a_2,a_2')\mtw(a_3,a_3')) \end{split}
\end{equation}
 by left-right symmetry.
\end{proof} 

Thus we can write twist products without parentheses.
The following notions have analogs in \cite[\S~3]{JuR1}.


\begin{definition} $ $ 
\begin{enumerate}
\item[(i)] 
A congruence $\Cong$ of $(\mcA,\mcA_0)$ is \textbf{semiprime} if
we have $\Phi_1=\Phi$ for any congruence $\Cong_1 \supseteq \Cong$  such that  $\Cong_1\cdot_{tw}\Cong_1\subseteq\Cong$. 
$(\mcA ,\mcA_0)$ is a \textbf{semiprime pair} if the trivial
congruence is semiprime.

In line with the standard terminology, when $\mcA$ is commutative, we write \textbf{radical congruence} instead of ``semiprime congruence.''

\item[(ii)] 
A congruence $\Cong$ of $(\mcA,\mcA_0)$ is \textbf{prime} if
we have $\Cong_1=\Cong$ or $\Cong_2=\Cong$  for any congruences $\Cong_1,\Cong_2 \supseteq \Cong$ such that  $\Cong_1\cdot_{tw}\Cong_2\subseteq\Cong$.
 $(\mcA ,\mcA_0)$ is a \textbf{prime pair} if the trivial
congruence is prime.

\item[(iii)]  
A congruence $\Cong$ of $(\mcA,\mcA_0)$ is \textbf{irreducible} if we have $\Cong_1=\Cong$ or $\Cong_2=\Cong$  for any congruences $\Cong_1,\Cong_2 \supseteq \Cong$ such that   $\Cong_1\cap \Cong_2 = \Cong$.
 $(\mcA ,\mcA_0)$ is an \textbf{irreducible pair} if the trivial
congruence is irreducible.

\end{enumerate}
\end{definition}

\begin{lem}\label{semip0}
A congruence $\Cong$ is semiprime iff $(b_1,b_1')\mtw (\mcA\times \mcA)\mtw (b_1,b_1') \subseteq \Cong$ implies $(b_1,b_1')\in \Cong$.

A congruence $\Cong$ is prime iff $(b_1,b'_1)\mtw (\mcA\times \mcA)\mtw (b_2,b_2') \subseteq \Cong$ implies $(b_1,b_1')\in \Cong$ or $(b_2,b_2')\in \Cong$.
\end{lem}
\begin{proof}
Easy consequences of Lemma~\ref{twass}, taking $\Cong_i$ to be the congruence generated by $\Cong$ and $(b_i,b'_i)$.
\end{proof}

The same proofs as in Jun and Rowen \cite[Proposition~3.14]{JuR1} can be used
to prove the following, for a $\tT$-pair $(\mcA,\mcA_0).$

\begin{itemize}
  \item 
 The intersection of semiprime congruences is semiprime.
 \item The union of a chain of congruences is a congruence.
 \item A congruence is prime if and only if it is semiprime and irreducible.
 \item 
(essentially \cite[Proposition 3.17]{JuR1}) For $\mcA$ commutative, define $\sqrt{\Cong}$ to be the set of elements $(a_1,a_2)$ of $\mcA \times \mcA$ having a twist-power $(a_1,a_2)^{\mtw m}$
(for some $m \in \mathbb{N}$)   in $\Cong.$ Then $\sqrt{\Cong}$ is a radical congruence and is the intersection of a nonempty set of prime congruences.
 \end{itemize}

 \begin{prop}\label{commpr}
[{essentially \cite[Proposition~3.17]{JuR1}}] Suppose $S$ is a multiplicative subset of~$\tT$, with $\zero\notin S.$  Then there is a prime congruence disjoint from $\hat S: = S\times \zero \cup \zero\times S$.
 \end{prop} 
 \begin{proof}
 $\hat S$ is a $\mtw$-multiplicative subset of $\hat \tT$, disjoint from
the diagonal congruence, so by Zorn's lemma there is a  congruence $\Cong$ of $\mcA$ maximal with respect to being disjoint from   $\hat S$. We claim that
$\Cong$ is prime. Indeed, if $\Cong_1 \mtw \Cong_2 \subseteq \Cong$ 
for  congruences $\Cong_1 , \Cong_2 \supseteq \Cong,$
then  $\Cong_1, \Cong_2 $ contain elements of $\hat S$, as does  $\Cong_1 \mtw \Cong_2,$ a contradiction. 
 \end{proof}
 
 
 Proposition~\ref{commpr} is enough to show  for $\mcA$ commutative
 that every radical congruence is the intersection of a nonempty set of prime congruences.
 We can generalize this result by using a famous trick of Levitzki.

\begin{theorem}\label{semip} A congruence $\Cong$ on a pair is  semiprime if and only if it is the intersection of prime congruences.
\end{theorem}
\begin{proof}
$(\Leftarrow)$ is obvious. Conversely, given $(s_1,s_1')\notin \Cong$
we need to find a prime congruence containing $\Cong$, but not containing $(s_1,s_1')$.
Inductively we build a subset $S \subset \mcA \times \mcA$ as follows:

We start with $(s_1,s_1') \in S.$ Given  $(s_i,s_i') \in S$, from Lemma~\ref{semip0}, there exists $(a_i,a_i')$ such that $$(s_{i+1},s_{i+1}') := (s_i,s_i')\mtw (a_i,a_i') \mtw (s_i,s_i') \notin \Cong.$$ Take a congruence $\Cong'$ maximal with respect to containing $\Cong,$ but not intersecting $S  = \{(s_i,s_i'): i \in \Net\}$. We claim that $\Cong'$ is prime. In fact, if $\Cong_1 \mtw \Cong_2 \subseteq \Cong'$ 
for  congruences $\Cong_1 , \Cong_2 \supset \Cong'$
then  $\Cong_1,\Cong_2 $ contain respective elements $(s_i,s_i')$ and $(s_j,s_j')$ of $S$ from the maximality of $\Cong'$.
If $i<j$ then $\Cong_1 \mtw \Cong_2$ contains $(s_{j+1},s_{j+1}')$, so $\Cong'$ contains $(s_{j+1},s_{j+1}')$, a contradiction. 
\end{proof}

\begin{defn} Define the \textbf{Krull dimension} of $(\mcA,\mcA_0)$ to be the maximal length of a chain of prime congruences of $\mcA$ containing $\Diag{\mcA}$.
\end{defn}

We note that the Krull dimension in the context of congruences was first introduced and studied by Jo\'o and Mincheva in \cite{JM}.

\begin{definition}
$\mcA$ satisfies the \textbf{ ACC (resp.~DCC) on congruences} if every ascending (resp.~descending) chain of congruences stabilizes. 
\end{definition}

We studied ACC on congruences in \cite[Proposition~3.15]{JuR1}, but there is a dearth of examples.

\begin{example}
  $\Net_{\operatorname{max}}$ does not satisfy ACC on congruences,
  since we can take $\Cong_i$ to be generated by
  $( 1, 2), \dots, (1,i).$
  $\Z_{\operatorname{max}}$ satisfies the  ACC on congruences but not DCC (even though it is a semifield),
   since we put $\Cong_i = \{ (m, m+i): m \in \Z\}$ (under classical addition).  $\Z_{\operatorname{max}}[\la]$ fails ACC since now we put $\Cong_i$ to be generated by
  $( \la, \la + 1), \dots, (\la^i, \la^i+1).$ The fact that $\mathbb{B}[x]$ does not satisfy ACC was first proved in \cite{AA}.
\end{example}

 \begin{defn}
Let $(\mcA,\mcA_0)$ be a $\tT$-semiring pair. 
\begin{enumerate}
    \item[(i)]
  The $\preceq$-\textbf{center} 
 $Z((\mcA,\mcA_0))$ of $(\mcA,\mcA_0)$ is 
 $\{ z \in \mcA : yz \preceq zy, \ \forall y \in \mcA\}$.
 \item[(ii)]
  $(\mcA,\mcA_0)$ is \textbf{commutative} if $\mcA =Z((\mcA,\mcA_0)).$
\end{enumerate}
   \end{defn}

   Note that $(\mcA,\mcA_0)$ commutative implies $y_1y_2 \preceq y_2 y_1$ for all $y_i \in \mcA.$
   This leads to the observation:
   
   \begin{lem}
       Suppose $\preceq = \preceq_0,$\footnote{Recall that $\preceq_0$ is a surpassing relation defined as follows: $x \preceq_0 y$ if and only if there exists $z \in \mcA_0$ such that $x=y+z$.} and $(\mcA,\mcA_0)$ has the property that $w + y_0  + y_0'  = w $
       implies $w+y_0  = w.$
       Then  $(\mcA,\mcA_0)$ is commutative if and only if  $\mcA$ commutative.
   \end{lem}
   \begin{proof}
  If $\mcA$ is commutative, then clearly $(\mcA,\mcA_0)$ is commutative. Conversely, for any $y_1,y_2 \in \mcA$, there exists $z,w \in \mcA$ such that 
 \[
 y_1y_2+z = y_2y_1, \quad y_2y_1+w = y_1y_2. 
 \]
It follows that $y_1y_2 = y_1y_2 z+w =  y_1y_2$,  showing that $y_1y_2 = y_1y_2 z+w =  y_2y_1$, proving that  $\mcA$ is commutative. 
\end{proof}


%


\subsection{Module pairs}

We fix a ground $\tT$-semiring pair $(\mathcal A, \mcA_0)$. 

\begin{defn}\label{modu2180} \begin{enumerate}
    \item 
A \textbf{module pair}
$\mathcal M : = (\mathcal M, {\mathcal N})$ over  $(\mathcal A,
\mcA_0)$ is a pair of $\mcA$-modules together with a given map
$\pMN: \mcN\to \mcM$ such that
\begin{equation}\label{eq: module}
\mcA_0 \mcM\subseteq \pMN(\mcN). 
\end{equation}
  \item A module pair
$ (\mathcal M, {\mathcal N})$ is $\tTM$-\textbf{admissible} if there is a $\tT$-action on $\tTM\subseteq \mcM$, and $\tTM \cup \{\zero_M\}$ spans $(\mcM,+)$ with $\tTM \cap \mcN = \emptyset.$
\end{enumerate}
 \end{defn}
 
Intuitively we view the module pair
$ (\mathcal M, {\mathcal N})$ as $\mcM/\mathcal N$. 

\begin{defn} A \textbf{surpassing relation} $\preceq$ on a $\tTM$-admissible module pair
$ (\mathcal M, {\mathcal N})$ is defined in analogy to
Definition~\ref{precedeq07}.

Given subsets $S_1,S_2 \subseteq \mcM, $ we write $S_1 \preceq S_2$ if for each $s_2\in S_2$ there is $s_1\in S_1$ for which $s_1\preceq s_2.$
 \end{defn}

\begin{defn}\label{modu218} $ $
\begin{enumerate}
    \item[(i)]
  A \textbf{homomorphism} of module pairs $F:(\mathcal M', \mathcal N')\to (\mathcal M, \mathcal N) $ is a
pair of module homomorphisms $f:\mathcal M'\to \mathcal M$ and $g: \mathcal N' \to \mathcal N,$ such that
$f(\phi_{\mathcal M',\mathcal N'})= \phi_{\mathcal M, \mathcal N} g.$
\item[(ii)]
A homomorphism $F$ is \textbf{monic} if $f^{-1}(\phi_{\mathcal M, \mathcal N}( \mathcal N)) = \phi_{ \mathcal M', \mathcal N'}( \mathcal N').$ In this case we say that $( \mathcal M', \mathcal N')$ is a  \textbf{submodule pair} of the module pair $(\mathcal M, \mathcal N) .$
\end{enumerate}
 \end{defn}

\begin{example}
 \label{reve}$ $ Here  are  examples of some of the notions.
 \begin{enumerate} \eroman
 \item 
 For $\mcA$ commutative and any $a\in \mcA$ the congruence $\Cong_a (\mcA)$ generated by $a$ is $\{(a_1a, a_2 a): a\in \mcA\}.$
    In the image of $\mcA$ under this congruence, $a$ is identified with $\zero a =\zero,$ so we have the natural homomorphism sending $a\mapsto \zero.$
    \item 
For any   module pair $(\mathcal M ,\mathcal N) $,  $(a \mathcal M, a \mathcal N) $ is a
submodule pair, for any $a \in \tT$.
\item Suppose  $\mcM$ is an $\mcA$-module, and $a\in Z(\mcA)$.
The congruence kernel  of the left multiplication homomorphism $\mcM \to a \mcM$ is $\mathcal K:=\{(y_1,y_2) \in \mcM \times \mcM: ay_1 = ay _2\}. $
$\mathcal K$ is an $\mcA$-module via the diagonal map, and we 
define $\pMK: K\to \mcM$ by $(y_1,y_2)\mapsto ay_1.$ In this
way we can identify $(\mcM,\mcN)$ with $a \mcM$, and $\mathcal K$ is the congruence  Kernel. (If need be, we
add on $\mcA_0 \mcM \times \mcA_0 \mcM.$)
\item  For a semiring pair $(\mathcal A, \mathcal A_0 )$
  and an index set $I$, ($\mathcal A ^{(I)}, \mcA_0 ^{(I)})$ is a
 module pair, where  $ \phi_{\mathcal A_0 ^{(I)},\mcA ^{(I)}}$ is defined by applying  $ \phi_{\mathcal A_0,\mcA}$ 
componentwise, and the $e_i$ are the usual vectors with $\one$ in
the $i$ position, comprising a  \textbf{base} $\{ e_i : i \in I\}$
of $\mathcal A ^{(I)}$. It is uniquely quasi-negated if $(\mathcal A, \mathcal A_0 )$ is uniquely quasi-negated, seen componentwise.
\item  For  $\mcM$ an $\mcA$-module, take $\mcN$ to be $\mcA_0 \mcM.$ 
 \item The following example from ~\cite[Definitions 5.4, 5.6]{CGR} played the main role in \cite{CGR}.
  Suppose  $(\mcA,\mcA_0)$ is a $\tT$-semiring pair with a  surpassing map $\preceq$, and $(\mcM, \mcN)$ is a module over $(\mcA,\mcA_0)$. A \textbf{symmetric bilinear form} $B: \mcM\times \mcN \to \mcA$ is a symmetric map satisfying
  $$B(a_1 v_1 + a_2 v_2, w)\succeq a_1 B(v_1,w) + a_2 B(v_2, w).$$
  A
  \textbf{quadratic pair} $(q,B)$ is a map $q: \mcM \to \mcA$ and  a symmetric bilinear form $B: \mcM\times \mcM \to \mcA$, satisfying
  $$q(av)\succeq a^2 q(v), \qquad 2q(v)= B(v,v), \qquad q(v+w) \succeq q(v)+q(w)+ B(v,w).$$
  The pair $(\mcA,\mcA_0)$  is \textbf{Clifford} if there is a quadratic pair $(q,B)$ on $\mcA$ for which $v_1^2 \succeq q(v_1)$ and $v_1v_2 + v_2 v_1 \succeq B(v_1,v_2)$ for all $v_i \in \mcA.$
\end{enumerate}
\end{example}

\begin{rem}\label{insert} If $\Cong$ is a congruence on a module $\mcM$
then there is a 1:1 set-theoretic map
$\Psi:\mcM/\Cong \to  \Cong$, given by taking a set-theoretic retraction $\psi: \mcM/\Cong \to \mcM$
and sending $[y] \mapsto (y,y)$. $\Psi$ will not be onto, but still is useful in estimating growth.
\end{rem}

 \subsection{Bases}
 \begin{definition}\label{free0}
Let $(\mcM,\mcN)$ be a module pair over $(\mcA, \mcA_0)$. 
\begin{enumerate}
    \item[(i)]
 A set $S \subseteq \mcM$ \textbf{$(\preceq)$-spans} $\mcM$
 if there are $a_i \in \mcA$ and $s_i \in S$ such that $ \sum_i a_is_i \preceq v$ for each $v\in \mcM.$ 

    \item[(ii)] 
A set $S \subseteq \mcM$ is $\preceq$-\textbf{independent} if $\sum a_i s_i \preceq \sum a'_i s_i$ in $\mcM$ for $a_i,a_i'\in \mcA$ implies each $a_i\preceq a_i'$ in $\mcA.$ 
\item[(iii)]
A \textbf{$(\preceq)$-base} of $(\mcM,\mcN)$ is an $\preceq$-independent set which $(\preceq)$-spans $\mcM$ over $\mcN.$ A  module pair with a $(\preceq)$-base is called \textbf{$(\preceq)$-free}.  
\end{enumerate}
 \end{definition}
 An obvious example: The unit vectors $\{e_1, \dots, e_n\}$ are a base for
 $\mcA^{(n)}$ over $\mcA_0^{(n)}$.
 More generally, $(\mcA^{I}, \mcA_0^{I})$ is free.
 
 \begin{rem}
  The universal algebraic definition of free,
  which shall call \textbf{universally free}, is that there is a
  set $\{b_i: i \in I\}$ such that for
  every module pair $(\mcM',\mcN')$ and $\{y_i: i \in I\}\subseteq \mcM'$ there is a unique morphism $\varphi:(\mcM,\mcN) \to (\mcM',\mcN')$ sending $b_i\mapsto y_i.$ 
  \end{rem}

\begin{prop}
Any universally free
  module  pair $(\mcM,\mcN)$ (over a set $I$) is isomorphic to $(\mcA^{I}, \mcA_0^{I})$.
\end{prop} 
\begin{proof}
 For a  module pair $(\mcM',\mcN')$, 
 define the module homomorphism $\varphi$ to the free module pair $(\mcM,\mcN)$ sending $\sum_i a_ie_i \to  \sum_i a_ib_i$. Clearly
 $\mcN \to \mcN'$.
 The set  $\{b_i: i \in I\}$ is seen to be a $(\preceq)$-base, by mapping it onto $\mcA^{(I)}$  by sending $\varphi : b_i\mapsto e_i$. 
  
  Note further that  $\varphi^{-1}(\mcA_0^{(I)}) = \mcN$. Indeed, if  $ \sum a_i\varphi ( b_i) = \varphi \sum (a_i b_i) \in \mcA_0^{(I)}$ then each $a_i \in \mcA_0 ,$ by independence. It follows that $(\mcM,\mcN)\cong (\mcA^{I}, \mcA_0^{I})$.
\end{proof}

  This proves that the  universally free
  module  pair is $(\preceq)$-free, and unique up to the cardinality of its $(\preceq)$-base.

\begin{remark} 
Given a submodule $(\mathcal M', \mathcal N')$ of $(\mathcal M, \mathcal N)$ and $F:(\mathcal M', \mathcal N')\to (\mathcal M, \mathcal N)$ with $f:\mcM' \to \mcM$ and $g:\mcN' \to \mcN$,
we can define $(\mathcal M, \mathcal N)/(\mathcal M', \mathcal N')$ to be the pair $(\mathcal M, \mathcal  M' )$ where $\phi_{\mcM,\mcN}: \mathcal M'\to \mathcal M$ is the map $f$. This is a useful tool in defining ``exact sequences'',
without congruences, 
namely $$(\mcK,\mcK) \to (\mcN,\mcK) \to (\mcM,\mcK) \to (\mcM,\mcN)\to (\mcM,\mcM),$$
which hints at homology, but we shall not pursue this intriguing direction in this paper. 
 \end{remark}
 
  \subsection{Extensions of pairs}

 We assume from now on that $(\mcA,\mcA_0)$ is commutative.
  
  \begin{defn} $ $
 \begin{enumerate}
     \item[(i)]
An \textbf{extension} of a $\tT$-semiring pair  $(\mcA,\mcA_0)$ is a  $\tT_\mcW$-semiring pair  $(\mcW,\mcW_0)$ where $\tT \subseteq \tT_\mcW. $  and $\mcW_0 = \mcA_0   \mcW.$ The extension $(\mcW,\mcW_0)$ is \textbf{centralizing} if $aw = wa$ for all $a\in \tT,\ w \in W.$

Note that $\mcW$ is generated by $\mcA $ and a subset $S\subset \tT_{\mcW};$ The   extension is \textbf{finitely generated} if $S$ can be taken to be     finite. Furthermore,   since any nonzero element of $\mcW$ is a finite sum of elements of $\tT_{\mcW},$ we will assume that $S \subset \tT_{\mcW}.$

  \item[(ii)]
An {extension} of a $\tT$-semiring pair  spanned by a finite number of elements is called a \textbf{finite extension}. 

  \item[(iii)] 
   A centralizing extension  $(\mcW,\mcW_0)$ of a $\tT$-semifield pair $(\mcA,\mcA_0)$ is called \textbf{affine} if there is a
   finitely generated centralizing extension  $(\mcW',\mcW_0)$ of $(\mcA,\mcA_0)$   with $\mcW' \preceq \mcW$.
 \end{enumerate} 
  \end{defn}
  
  In the natural examples, one would expect $\mcA_0 = \mcA \cap \mcW_0,$ but we do not see how to guarantee this.
 
 \begin{prop}\label{spans}
 If a semialgebra extension $(\mcW,\mcW_0)$ is a $(\preceq)$-free module over   $(\mcA, \mcA_0)$ with $(\preceq)$-base $B$ and $H$ is a sub-semialgebra of $\mcA$ over which $B$ still $(\preceq)$-spans $\mcW$, then $\mcA \succeq H$, cf.~Definition~\ref{precedeq07}.
 \end{prop}
 \begin{proof}
 For any $c\in \mcA$ write $cb_1 \succeq\sum h_i b_i.$ Then by definition of $(\preceq)$-base, $c \succeq h_1 \in \mcA.$
 \end{proof}

 \subsection{Function pairs and polynomials}\label{pol}
 
 Some of this material is reminiscent of \cite{JuR1}, which handled the commutative situation.
 
 \begin{defn} Given a  set $S$, the \textbf{support}   of a function $f: S \to \mcA$ is $\{ s \in S : f(s) \ne \zero\}.$
We define $\mcA^S_{< \infty}$ to be the set of functions from $S$ to $\mcA$ with finite support.
 If $\mcA$ is a $\tT$-module, we view 
 $\mcA^S_{< \infty}$ as a module over  $\tT_S$, defined as
 the set of functions $f: S \to \mcA$ whose support is a singleton $\{s\}$ with $ f(s) \in \tT\}$.
 
 When   $\mcA$
 is a bimagma, we define the \textbf{convolution product} as follows:
 \begin{equation}
 f*g(s) = \sum_{uv=s}f(u)g(v) \textrm{ for }f,g\in \mcA^S_{< \infty}.
 \end{equation}
For a set of indeterminates $\Lambda = \{\la_i : i\in I\},$
 the  \textbf{polynomial magma} $\mcA[\Lambda]$ is  $(\mcA^S_{< \infty})$ where $S=\Net^I,$ identifying $(m_i)$
with $\prod \la_i^{m_i}.$ 
 We write $f(\la_1,\dots\la_m)$ as a typical polynomial.

 Given $\{ b_i: i\in I \} \subseteq \mcA,$ and $\La = \{\la_i: i \in I\},$ there is a unique homomorphism $\mcA [\Lambda]\to \mcA$ sending $\la _i\mapsto b_i.$ We write the image of
 $f(\la_1,\dots\la_m)$ as $f(\mathbb{b})$, where
 $\mathbb{b} = \{b_1, \dots, b_m).$
 
 We define the \textbf{convolution product} as follows:
 \begin{equation}
 f*g(s) = \sum_{uv=s}f(u)g(v) \textrm{ for }f,g\in \mcA^S_{< \infty}.
 \end{equation}
For a set of indeterminates $\Lambda = \{\la_i : i\in I\},$

 Given $\{ b_i: i\in I \} \subseteq \mcA,$ and $\La = \{\la_i: i \in I\},$ there is a unique homomorphism $\mcA [\Lambda]\to \mcA$ sending $\la _i\mapsto b_i.$

  \end{defn}

\begin{prop}\label{polp}
If $(\mcA ,\mcA_0)$ is a semiprime $\tT$-pair then $(\mcA^S_{< \infty},{\mcA_0^S}_{< \infty})$ is a semiprime $\tT_S$-pair. 
\end{prop}
\begin{proof}
One checks it pointwise in \cite[Proposition~4.2]{JuR1}.
\end{proof}
 
 On the other hand, the prime analog proved in
 \cite[Theorem~4.6]{JuR1} required a Vandermonde argument that relies on commutativity. One needs that two functions agreeing on ``enough'' points are the same. 
 
 \begin{rem} As  an illustration of  Proposition~\ref{polp},
 given an admissible $\tT$-pair $(\mcA,\mcA_0)$, we have the  $\tT_{\Lambda}$-\textbf{polynomial pair} $(\mcA[\Lambda],\mcA_0[\Lambda]),$ where $\tT_{\Lambda}$ denotes the monomials with coefficients in $\tT.$
 
 This pair is not shallow, since $\la + \one^\circ \notin \mcA_0[\Lambda] \cup \tT_{\Lambda}$. One could try to remedy this by taking $\tT[\Lambda]$ instead of $\tT_{\Lambda},$ but  $\tT[\Lambda]$ is never a monoid when $(\mcA,\mcA_0)$ satisfies Property N. For if $a+a'\in \mcA_0$ for $a,a'\in \tT$, then $(\la +a)(\la +a') = \la^2+(a+a')\la + aa'. $
 
 On the other hand we could have obtained a shallow pair satisfying Property N by taking $\mcA[\Lambda]_0$ to be $\mcA_0[\Lambda] \cup \{\text{polynomials which are a sum of at least two monomials}\}.$
 \end{rem}

 One can also define polynomials symbolically, but there are many examples in tropical mathematics of differing polynomials which agree as functions,
 such as $\la^2 + a\la + 4$ for all $a<2$. This will impact on our discussion of algebraicity.

 \section{Paired versions of classical theorems}

We generalize some results from \cite{JuR1}.

 \begin{theorem}[Artin-Tate for pairs]\label{AT} Suppose $\mcW$ is an affine 
 semialgebra over $\mcA$
 and has a finite $(\preceq)$-base $B$ over a central semialgebra $\mcK \subset \mcW$. Then $\mcK$ is affine over $\mcA.$ 
 \end{theorem} 
 \begin{proof} 
 As in \cite[Theorem~4.26]{JuR1}, write $ \mcA [y_1,\dots, y_n]\preceq \mcW $ and $$\sum \alpha _{ijk} b_k \preceq b_i b_j  , \qquad  \sum \alpha _{uk} b_k\preceq y_u , \quad 1 \le u\le n.$$
 Let $H$ be the sub-semialgebra generated by all the $\alpha_{ijk}$ and $\alpha_{uk}.$ Then $B$ also $(\preceq)$-spans
 each $y_i$ over $H$, and thus $\sum \{Hb_i : b_i \in B\}$ is a subalgebra
 over which $B$ $(\preceq)$-spans $\mcW,$ and thus
 by Proposition~\ref{spans}  $ H \preceq \mcK,$ so $\mcK$ is affine by definition. 
 \end{proof}
 
   \subsection{Fractions}
   
   Fractions have already been studied for monoids, cf. \cite{Lj} for instance,
   and here we present the paired version,
  taking the analog from \cite[\S 3.1]{Row06}.
  
  \begin{defn} An element  $s\in\tT$ is \textbf{left regular} if $b_1s = b_2s$ for $b_i\in \mcA$ implies $b_1 = b_2,$ $s\in\tT$ is \textbf{left $\preceq$-regular}  and  $b_1s \preceq b_2s$ for $b_i\in \mcA$ implies $b_1 \preceq b_2$. \textbf{Right regular} is analogous, and \textbf{regular} means left and right regular. \end{defn}

   \begin{defn}\label{denomset}
   A $\tT$ satisfies the   \textbf{(left) Ore condition} with respect to a subset  $S$ of regular elements of $\tT$ if:
   \begin{itemize}
       \item  For any $b\in \mcA$ and $s\in S$ there are $b'\in \mcA$ and $s'\in S$
   such that $s'b = b's$.
    \item If $b_1s = b_2 s$ then there is $s'\in S$ with $s'b_1 = s'b_2.$
   \end{itemize}
   \end{defn}

The main tool is from monoids.
\begin{lem}
Given a $\tT$-semiring $ \mcA$  with $\tT$ satisfying the     Ore condition  with respect to $S$,  one can define an equivalence relation on $\mcA\times S,$ by $(b_1,s_1) \equiv (b_2,s_2)$ iff there are $a_1, a_2\in \tT$ for which $ a_1 b_1= a_2b_2$ and  $ a_1 s_1= a_2s_2\in S$.
\end{lem}
\begin{proof} Reflexivity and symmetry are clear. For transitivity
we start with a sublemma:

If $(s_1,b_1)\sim (s_2,b_2)$ and $c_1s_1 = c_2s_2\in S$ for $c_i\in\tT$, then there is $a\in \tT$ such that $ac_1s_1 = ac_2s_2 \in S$ and $ac_1b_1 = ac_2b_2.$ 
 
 \textbf{Proof of sublemma}: 
Take $a_i
\in \tT$ with $ a_1 b_1= a_2b_2$ and  $ a_1 s_1= a_2s_2\in S$.
Take $s\in S$, $y\in \mcA,$  such that $ ya_1 s_1 = s c_1 s_1.$
The Ore condition gives $s_1'\in S$
with $ s_1'y a_1  =  s_1' s c_1$. Then
$$ s_1' y c_2 s_2 = s_1' y  c_1 s_1=  s_1' y a_1 s_1 =  s_1' y a_2 s_2.$$
Hence there is $s_2'$ for which
$s_2's_1' c_2s_2 =s_2' s_1' y a_2, $
and we take $a = s_2's_1' c_2$
and check that  $ac_1s_1 = ac_2s_2 
\in S$ and $ac_1b_1 = ac_2b_2$ as desired. 

Transitivity now follows easily as in \cite[p.~350]{Row06}.
\end{proof}

\begin{thm}
\label{OreP}  
We take the equivalence of the lemma, where $\tT$ satisfies the   {(left) Ore condition} with respect to   $S$.
\begin{enumerate}
\item[(i)]  
Write $S^{-1}\mcA$ for  $\{s^{-1}b:= [(b,s)] : b\in \mcA, s\in S\},$  $S^{-1}\mcA_0$ for  $\{s^{-1}b:= [(b,s)] : b\in \mcA_0, s\in S\},$ and $S^{-1}\tT$ for  $\{s^{-1}a:= [(a,s)] : a\in \tT, s\in S\}.$

\item[(ii)] Any two elements $s_1^{-1}b_1$ and $s_2^{-1}b_2$ can be written with a common denominator $s$ where we have $s = s' s_1 = b' s_2 \in S$ for suitable $s'\in S$ and $b' \in \mcA$
\item[(iii)]    $S^{-1}\mcA$ has well-defined addition given by $s^{-1}b_1 +s^{-1}b_2 = s^{-1} (b_1 +b_2)$ and multiplication
$s^{-1}b_1 s^{-1}b_2 =
{s's}^{-1}b'b_2$  where $b's_1 = s' b_1$ from Definition~\ref{denomset}.
\item[(iv)]  
$(S^{-1}\mcA,S^{-1}\mcA_0)$ is an $S^{-1}\tT$-pair.
 When $\tT$ is multiplicatively cancellative, the  $\tT^{-1}\tT$-pair $(\tT^{-1}\mcA,\tT^{-1}\mcA_0)$ is a $\tT^{-1}\tT$-semifield pair. \end{enumerate}\end{thm}
\begin{proof}
(i) As in the classical case, following \cite[Theorem~3.1.4]{Row06}. 

(ii) $s^{-1}(s'b_1) \equiv s_1^{-1}b_1$ and $s^{-1}b'b_2 \equiv s_2^{-1}b_2$.

(iii) is as in  \cite[p.~351]{Row06}, using the same argument for well-definedness of multiplication using the trick used in the proof of (i) to avoid the use of negation.

(iv) Straightforward, using the definition of regularity given here.
\end{proof}
We call  $(\tT^{-1}\mcA,\tT^{-1}\mcA_0)$ of (iv) the    \textbf{pair
of fractions of the pair  $ (\mcA,\mcA_0)$}.
In all of our applications, the regular set $S$ will be central, and thus automatically Ore.
  
   \subsection{Integral extensions}
  \begin{defn} $ $
\begin{enumerate}
    \item[(i)]
  Suppose $(\mcW,\mcW_0)$ is a centralizing extension of a commutative $\tT$-semiring pair $(\mcA,\mcA_0)$. An element $y\in \mcW$ is \textbf{integral over} $(\mcA,\mcA_0)$, if 
   there are $a_0,\dots, a_{n-1}\in \mcA$ such that $ \sum_{i=0}^{n-1}a_iy^i \preceq y^n .$  The minimal such $n$ is called the \textbf{degree} of $y.$ An element 
   $y\in \mcW$ is  $\tT$-\textbf{integral} if each
   $a_i\in\tT$.
  \item[(ii)]
 $(\mcW,\mcW_0)$ is an \textbf{integral extension} of $(\mcA,\mcA_0)$ if each element of $\mcW$ is integral over $(\mcA,\mcA_0)$.
\end{enumerate}

     \end{defn}
\begin{rem}\label{int2}
 It follows at once that if $y$ is integral over $(\mcA,\mcA_0)$ then $(\mcA[y],\mcA_0[y])$ is a finite centralizing extension of $(\mcA,\mcA_0)$,
  spanned by $\one, y, \dots, y^{n-1}.$
\end{rem}
       
We the following  property, to provide the converse of Remark~\ref{int2}.
 
     \begin{defn}\label{reversi} $ $
\begin{enumerate}
    \item[(i)]
   $(\mcA,\mcA_0)$ satisfies \textbf{reversibility} for $a$ if
     $b + a \succeq \zero$  implies $b \succeq a,$ $\forall b\in \mcA$.
     In this case, we  say that $a$ is \textbf{reversible}.
     \item[(ii)]
     We  say that $a$ is \textbf{power-reversible} if $a^n$ is reversible for each $n$.
     \item[(iii)]
     $(\mcA,\mcA_0)$ satisfies \textbf{tangible reversibility} if
      it satisfies reversibility for each  $a \in \tT$.     
\end{enumerate}
  \end{defn}
     
      \begin{defn}\label{reversine}
Suppose that $(\mcA,\mcA_0)$ has 
       a negation map $(-)$.
\begin{enumerate}
    \item[(i)]
    $(\mcA,\mcA_0)$ satisfies \textbf{$(-)$reversibility} for $a$ if
     $b (-) a \succeq \zero$  implies $b \succeq a,$ $\forall b\in \mcA$. In this case, we shall say that $a$ is \textbf{$(-)$-reversible}.
     \item[(ii)]
  We  say that $a$ is \textbf{$(-)$-power-reversible} if $a^n$ is $(-)$-reversible for each $n$.
  \item[(iii)]
 $(\mcA,\mcA_0)$ satisfies \textbf{tangible $(-)$-reversibility} if it satisfies $(-)$-reversibility for each  $a \in \tT$.  
\end{enumerate}       
    \end{defn}
     
     \begin{lem} Tangible reversibility  holds in the following settings: system, supertropical system, hypersystem, and symmetricized system.
     
      Tangible $(-)$-reversibility  holds in the following settings: classical system, supertropical system, and  hypersystem. 
     \end{lem}
     \begin{proof} Suppose   $b (-) a \succeq \zero$.
     
     The classical case is obvious, since then $b-a = 0$ implies $a=b.$
     
     For the supertropical, 
     $(-) = +$; if $b \ne a$ and $b \in \mcA^\circ$ then 
     $b+a = b.$
     
     For a hypersystem built on a hypergroup $H$, $a\in H$ and  $a' - a = 0$ for some element $a'\in b$, so $a = a'\in b.$
     
     For a symmetricized system, let $b = (b_1,b_2)$ and $a = (a,0)$ Then $b +a =  (b_2,b_1)+ (a,0) = (b_2+a, b_1)$, so $b_1= b_2+a$ and $(a,0) + (b_1,b_2) = (b_1,b_2)$
     \end{proof}
     
     Clearly tangible $(-)$-reversibility implies power $(-)$-reversibility for each  $a \in \tT$. In this case we can define the \textbf{negated determinant}
  $\det(A)$ of a matrix $A=(c_{ij})$ 
to be the following:
\begin{equation}
\det(A):=  \sum _{\pi\in S_n} ((-) \one)^{\sgn(\pi)} c _{i_1 \pi(i_1)}   c_{i_2 \pi(i_2)} \dots  c_{i_n \pi(i_n)}.   
\end{equation}

  \begin{prop}
  Suppose $(\mcA[y],\mcA_0[y])$ is  a  finite centralizing extension of a commutative  pair $(\mcA,\mcA_0)$, and $y$ is reversible or $(-)$-reversible.
  Then $y$ is integral over $(\mcA,\mcA_0)$.
  \end{prop} 
  \begin{proof}
  Write $\mcA[y] \succeq \sum_{i=0}^n \mcA y_i.$ Then, for each $1\le i \le n,$
  $yy_i \succeq \sum c_{ij} y_i$ for suitable $c_{ij}\in\mcA,$ implying the matrix $$A = \left(\begin{matrix} c_{11}(-) y & c_{12} & \dots &  c_{1n} \\ c_{21} & c_{22}(-) y  & \dots & c_{2n} \\ & &  \ddots & \\  c_{n1} & c_{n2} & \dots &  c_{nn}(-) y
\end{matrix}\right)$$
satisfies $A v \succeq \zero,$ for $v$ the column vector $(y_1,\dots y_n),$ and letting $\adj{A}$ denote the (negated) adjoint matrix as in \cite[Definition~8.2]{Row21}, we have
$\det(A)v = \adj{A} A v \succeq \zero$. Hence $\det(A)y_i\succeq \zero$ for each $i$, and thus $$\det(A)\mcA[y]\succeq \sum _i \det(A)\mcA y_i \succeq \zero, $$ implying $\det(A) \succeq\zero.$ But opening $\det(A)$ and reversing $y^n$ if necessary
gives an integral relation for $y$ over $(\mcA,\mcA_0)$.
\end{proof}

      \begin{defn}
   Suppose $(\mcW,\mcW_0)$ is an extension of $(\mcA,\mcA_0)$.  An element $y\in \mcW$ is \textbf{algebraic over} $(\mcA,\mcA_0)$ if there
 are $a_i \in \mcA$ such that $ \sum_{i=0}^n a_i y^i \in \mcW_0$. The minimal such $n$ is called the \textbf{degree} of $y$. The $a_i$ are called the \textbf{coefficients} of $y$.
  \end{defn}
    
   Unfortunately, an algebraic element over a $\tT$-semifield pair need not be $\tT$-integral.
    Here is a special case where $\tT$-integrality holds.
   
   \begin{prop}\label{algint}$ $
  \begin{enumerate}\eroman
      \item Every  power-reversible or $(-)$-power-reversible  algebraic
   element over a $\tT$-semifield pair, with leading coefficient in $\tT,$ is $\tT$-integral.
   
   \item If $s   = \sum b_i a^i \in \tT$ and the pair $(\mcA, \mcA_0)$ is shallow, then we can take all the coefficients
   $b_i$ to be in $\tTz$.
  \end{enumerate}
   \end{prop}
   \begin{proof} (i)
 This is clear as we can divide through by the leading coefficient,
 and then apply reversibility to the leading term. 
   
   (ii) Write $s = s_1 + s_2,$ where $s_1 = \sum b_ia^i$ with $b _i \in \tT$ and 
   $s_2 = \sum b_i'a^i$ with $b' _i \in \mcA_0.$ Then $s\succeq s_1,$ implying $s=s_1$  since $s$ is tangible, cf. Lemma~\ref{sh2}.
   \end{proof}

\subsection{Hilbert Nullstellensatz}

\begin{definition} \label{nondeg} 
[Another version of algebraicity] $ $
\begin{enumerate}
    \item[(i)]
A polynomial $f$ is \textbf{tangible} if  $f = \sum a_{i} \lambda^i,$ with $a_i \in \tT$.

\item[(ii)] $ y\in \mcW$ is \textbf{transcendental} if for any polynomials
   $f_1, f_2$, if $f_1(y)\succeq f_2(y)$ then $f_1(b)\succeq f_2(b)$ for all $b \in \mcA;$ 
$ y\in \mcW$ is \textbf{congruence algebraic} over the pair $(\mcA, \mcA_0)$ if
$y$ is not transcendental.
 
\end{enumerate}
 \end{definition}
   \begin{rem}\label{tvp}
       If $(\mcA,\mcA_0)$
       is shallow and $\preceq$ -nondegenerate, then every tangible polynomial takes on a tangible value. Indeed for all $a\in \tT,$ if  $f(a)\notin \tT$  then $f(a) \in \mcA_0$, so $f\in \mcA_0[\La].$
       
       This can be coupled with Lemma~\ref{sh2}.
   \end{rem}

   \begin{prop}\label{sN}
   Suppose a shallow semiring pair  $(\mcA, \mcA_0)$ is  $\preceq_0$ -nondegenerate, and
  \  $(\mcW, \mcW_0)$  is a centralizing extension of $(\mcA,\mcA_0)$, with $y\in \tT_\mcW$ such that $(\mcA[y], \mcA_0[y])$ is tangibly separating. Let 
  $\tT' = \{ a y^i : a \in \tT,  i \in \Net\}.$
  Let $(\mcK, \mcK_0)$ be  the $\tT'$-semifield of fractions of  $(\mcA[y], \mcA_0[y])$. If  $(\mcK, \mcK_0)$  is affine, then
     $y$ is  congruence algebraic over the pair
  $(\mcA, \mcA_0)$.
   \end{prop}
   \begin{proof}
 Write $(\mcK, \mcK_0)\succeq (\mcA [a_1, \dots, a_m], \mcK_0),$  with each $a_i\in \tT_\mcW.$ Write each $a_i = s^{-1} f_i(y),$
$1\le i \le t,$ where $s\in \tT_\mcW,$ cf. Theorem~\ref{OreP}(2).  Note that $\deg g >1,$ or else $\mcK = \mcA[y]$ and we are done.    Write $s = g(y),$ where by Proposition~\ref{algint} we may assume that  $g\in \mcA[\la]$ has coefficients in $\tT$. By $\preceq$-nondegeneracy we have $g(a)$ tangible for some $a\in \tT'.$ Since  $(\mcA[y], \mcA_0[y])$ is tangibly separating, we have $a'\in \tT'$ such that $y+a'\in \tT'$ and $a+a'\in \mcK_0.$
Hence $(y+a')^{-1}\in \mcK= {\tT'}^{-1}\mcA[y]$
so $(y+a')^{-1} \succeq f(y) $ for some
   $f \in \mcA[\lambda].$ Then
  $(y +a')^{-1}  \succeq  s^{-k}f(y) = g(y)^{-k}f(y)$,
  so $g(y)^k \succeq (y +a')f(y)$. If $y$ were transcendental then $g(a)^k \succeq (a +a')f(a)$, i.e., $g(a)^k \in \mcA_0,$
  contrary to $g(a)$ tangible. Hence $y$ is   congruence algebraic.
   \end{proof}
   
    Theorem~\ref{sN} could be viewed as a ``baby Nullstellensatz,'' and we would like the conclusion that $y$~is integral over the pair
  $(\mcA, \mcA_0)$. But there   are counterexamples.
  
  \begin{example}\label{nonin}\begin{enumerate}\eroman
      \item 
   Take any pair  $(\mcA, \mcA_0)$ and adjoin an additively absorbing element $\infty$ to $\tT$, in the sense that
   $\infty^n + a = \infty^n$ for all $a\in \tT$ and all $n.$ Then for any tangibly monic polynomial $f$ of degree $n$, $f(\infty) = \infty^n,$ implying the $\tT$-semifield  pair of fractions of $(\mcA[\infty],\mcA_0[\infty])$ is affine (generated by $\infty^{-1}$).
        \item Call $\mcA$ \textbf{weakly bipotent} if  $a,a' \in \tT$ implies $a+a' \in \{a, a'\}$ or $a^2 = (a')^2.$
        (This is implied by $(-)$-bipotence in systems, cf.~\cite{Row21}, and often is the case in tropical mathematics.) Then,
        taking $a' = a^2,$ either
        $a^2+a = a^2$ or  $a^2+a = a$ or 
         $a^2= a^4.$ But we get the reverse for $a^{-1}$ instead of $a$.  Thus for any such $\tT$  not satisfying the identity $x^2 = x^4,$ any invertible $a\in \tT$  is $\tT$-congruence algebraic, but not necessarily $\tT$-integral.
        \end{enumerate} 
  \end{example}
   
   These examples weaken the strength of the next result.
   
\begin{thm}\label{NullS}
Suppose $(\mcW,\mcW_0)$ is a commutative affine $\tT$-semifield pair over a   shallow, $\preceq$-nondegenerate semifield pair $(\mcA,\mcA_0)$, and also having the property that  $\tT$-{congruence algebraic} implies $\tT$-integral. Then $(\mcW,\mcW_0)$ is $\tT$-integral.
\end{thm}
\begin{proof}
This can be easily proved by induction on $n, $ where $\mcW = \mcA[y_1, \dots, y_n].$

Write $\mcW \preceq \mcA [y_1]([y_2,\dots, y_n]),$
 and let $(\mcK, \mcK_0)$ be the $\tT$-semifield
of fractions of the pair  $ (\mcA[y_1],\mcA[y_1]\cap \mcW_0)$. By Theorem~\ref{AT},  $(\mcK, \mcK_0)$ is affine. Hence, by Theorem~\ref{sN}, $\mcK \preceq \mcA[y_1],$
and is finite over $(\mcA,\mcA_0)$. But by induction,
$(\mcW,\mcW_0)$ is integral over $(\mcK, \mcK_0)$, and thus over $(\mcA,\mcA_0)$.
\end{proof}

Unfortunately, from the standpoint of tropical mathematics,
 the property that  $\tT$-{algebraic} implies $\tT$-integral
 is not compatible with the tropical viewpoint, by 
 Example~\ref{nonin}. But it is all we have.
 
\section{Growth in semialgebras}

Growth in algebraic structures has been an active area of study in the last 30 years, cf.~\cite{KrL}, mostly for groups and algebras, although recently growth in semigroups and other algebraic structures has been investigated in \cite{BZ,Gr,Sh}. Some of the basic properties carry over to semialgebras, as we review here.

\begin{definition}
Let $f$ be a function from the set of algebraic pairs to natural numbers. We say that $f$ is \textbf{sub-additive} (resp.~\textbf{sub-multiplicative}) if the following holds: for any pairs  of objects $(A,B)$ and $(B,C),$  
\[
f(A,C) \le f(A,B)+f(B,C) \quad (\textrm{resp.~} f(A,C) \le f(A,B)f(B,C))
\]
when equality holds we say that $f$ is \textbf{additive} (resp.~\textbf{multiplicative}).
\end{definition}

\subsection{Growth in a pair}

 Let $[\mcM:\mcN]$ denote the minimum number of elements need to generate $\mcM$ over $\mcN$, called the \textbf{rank}. The rank is sub-multiplicative:  If $\mcN \subseteq \mcN'\subseteq \mcM$, then $[\mcM:\mcN] \leq [\mcM:\mcN'][\mcN':\mcN].$ In situations where equality holds (such as in the classical situation for modules over Artinian rings), one can develop 
 dimension theories.
 
In this subsection we assume that $(\mcW,\mcW_0)$ is an affine centralizing extension   over a $\tT$-semifield pair $(\mcA,\mcA_0)$. In the non-relative case we take $\mcW_0 = 0.$ 

Now, we introduce a notion of the growth rate. We take a generating set $\{ a_1, \dots, a_{t'}\}$ of $\mcW_0$ over~$\mcA$ (which is empty when $\mcW_0 = 0$), which we extend to a generating set $\{ a_1, \dots, a_{t}\}$ of $\mcW$ over $\mcA$; we put $V' = \sum_{i=1}^{t'}\mcA a_i$ and $V = \sum_{i=1}^{t} \mcA a_i$. Note that $V=V'$ when $\mcW_0 = 0.$
We have the filtration $\mcW_k = \sum_ {i=1}^k V^k$ (which is just $V^k$ if $1 \in V$) of~$\mcW$, $1 \le k <\infty,$ and ${\mcW_0}_k = \sum _{i=1}^k {V'}^k$. We define the following numbers:
\begin{equation}\label{eq: 12}
d_k:=[\mcW_k:\mcW_{k-1}+{\mcW_0}_k].    
\end{equation}

\begin{defn} 
The \textbf{growth rate} of $\mcW$ (with respect to $\mcW_0$) is the sequence $\{d_1,d_2,\dots\}$, where $d_k$ is as defined in \eqref{eq: 12}. 
\end{defn}

If we start with a different set of generators then we may get a different growth rate $\{d_1',d_2',\dots\},$ but they are \textbf{equivalent} in the sense that there are
numbers $m_1,m_2$ such that $d_k' \le  m_1d_{m_2k}$ and $d_k \le m_2 d_{m_1k}'$ for all $k.$ Since $\tTz{}$ $(\preceq)$-spans $\mcW$ we may take the generating set in $\tTz.$
  
\begin{remark}
The basic definitions do not involve subtraction, 
and thus many classical results go over directly to semialgebras, and then to semialgebra pairs, almost word for word. So we shall quote proofs of standard results.

\end{remark}
\begin{example}
 The basic examples are analogous to the ones found in any standard algebra text, such as \cite[Chapter 17]{Row08}.
 \begin{enumerate}
     \item The \textbf{free semialgebra} over a commutative semiring is the monoid semialgebra over the word monoid $X= \{x_1,\dots, x_t\}$ in $t$ letters. The words in $X$ are a base,
     so $d_k$ is the number of words of length $k,$ which is $t^k.$ Obviously this exponential growth rate is the largest possible growth rate for a semialgebra.
     
     \item When $\mcA$ is finitely spanned over $\mcA_0$,
     the $d_k$ are  bounded. If  $\mcA$ is not finitely spanned over $\mcA_0$, then $d_k\ge k.$
     
      \item The \textbf{free commutative semialgebra}
      is the polynomial semialgebra $\mcA [\lambda_1, \dots \lambda_t].$ When $1 \in \mcA,$ $d_k = \binom{k+t}{t}\le k^t,$ which has polynomial growth.
 \end{enumerate}
\end{example}
 
Here is a semialgebraic analog of a theorem of Jategaonkar (that any domain of subexponential growth is Ore). We say that a $\tT$-pair $(\mcA,\mcA_0)$
is a \textbf{$\tT$-semidomain pair} if every element of $\tT$ is regular over $\mcA_0$.

\begin{prop}\label{Jat} Any $\tT$-semidomain pair $(\mcA,\mcA_0)$ with subexponential growth has the property that for any $a_1,a_2\in \tT$ there are $b_1,b_2\in \mcA\setminus \mcA_0$ such that $b_1a_1+b_2a_2\in \mcA_0$.
\end{prop}

\begin{proof} Suppose $a_1, a_2\in \tT$.
Then $f(a_1,a_2)\in \mcA_0$ for some
polynomial $f = g \lambda_1+h\lambda_2\in \tT[\la_1,\la_2].$ Take such $f$ of minimal total degree. Then $\deg g, \deg h < \deg f,$ so
by hypothesis $g(a_1,a_2), h(a_1,a_2) \notin \mcA_0$,
and thus $g(a_1,a_2)a_1, h(a_1,a_2) a_2 \notin \mcA_0 $. Thus we can take
$b_1= g(a_1,a_2)$ and $b_2 =h(a_1,a_2). $
\end{proof}

The conclusion of the proposition could be interpreted as saying that $(\mcA,\mcA_0)$  satisfies the left Ore condition with respect to~$\tT$.

\begin{defn}
We define the \textbf{Hilbert Series} of $(\mcW, \mcW_0)$ to be $\sum _{k=1}^\infty d_k\lambda^k \in \mathbb{N}[\![\lambda ]\!].$\footnote{See \cite{EG} for an alternative definition.}
\end{defn}

This is completely analogous to the standard algebraic situation, and
can be viewed in terms of the \textbf{graded pair} $\oplus (\mcW_k, \mcA_{k-1}+{\mcA_0}_k).$ Thus growth, Gelfand-Kirillov dimension and Hilbert Series of semialgebras are closely related to semigroups, and  Example~\ref{Making} is relevant. Shneerson~\cite{Sh}, Smoktunowicz~\cite{Sm}, and Greenfeld~\cite{Gr} have interesting semigroup examples of varied growth.

\begin{ques} Suppose $\mcA$ is commutative and $\Net$-graded. Is the Hilbert series of a finitely spanned $\mcA$-module $\mcM$ (over $\mcA_0$) rational?
\end{ques}

\begin{lem}
When $\mcA_0=0,$ the Hilbert Series of the polynomial semialgebra $\mcA[\la_1,\dots, \la_t]$ is $\frac{1}{(1-\la)^t}$.
\end{lem}
\begin{proof}
Just as in the classical case, cf.~\cite[Example~17.39]{Row08}.
\end{proof}

\begin{remark}
The referee has suggested that the Hilbert series of a tensor product is the product of the Hilbert series.
\end{remark}
 
Define the \textbf{Gelfand-Kirillov} dimension $\GKdim (\mcA,\mcA_0)$ of $(\mcA,\mcA_0)$
to be $$\limsup{\log _k [\mcA_k:{\mcA_0}_k]}.$$
\begin{lem} \begin{enumerate}
    \item 
$ \GKdim{(M_n(\mcA),\mcA)}=0.$ More generally,  $ \GKdim{(\mcA,\mcA_0)}=0$ whenever $\mcA$ is a finite extension of $\mcA_0$.

\item  $ \GKdim{(\mcA[\lambda_1,\dots, \lambda_n],\mcA[\lambda_1,\dots, \lambda_n]_0)}= \GKdim{(\mcA,\mcA_0)}+n$, for any $n$.
\end{enumerate}
\end{lem}
\begin{proof}
(i) The dimensions are bounded.
(ii) The same proof as for the classical case,
cf.~\cite[Example~17.47]{Row08}, for example.
\end{proof}
\begin{ques} For $(\mcA,\mcA_0)$ commutative,  is $\GKdim (\mcA,\mcA_0)$ always an integer, and does it equal the Krull dimension?
\end{ques}
 
 \section{Appendix A: Roots  of a polynomial}
 
 Since algebra often serves as a tool for geometry, we use this appendix to lay out the geometric concepts arising from pairs.

\begin{defn}
A $\preceq$-\textbf{root} of a polynomial $ f(\la_1,\dots,\la_m)$
is a tuple $(b_1, \dots, b_m) \in \mcA^{(m)}$ such that $f(b_1, \dots, b_m)\in \mcA_0.$
\end{defn}.

This definition is natural, consistent with \cite{IR,IRMatrices}. Then for $Z \subset \mcA^{(n)}$ one would take $\La = \{\la_1, \dots, \la_n\}$ and
 $\mathcal{I}(Z) = \{f \in \mcA[\La] : f(z) \in \mcA_0, \ \forall z \in Z.$ 
 
 The difficulty in tying this in to algebra is that $\mathcal{I}(Z) $ is an ideal of $\mcA[\La]$, not a congruence.
  
  Accordingly we take instead points of $\mcA^{(n)}\times \mcA^{(n)}$ and for $(z_1,z_2) \in  \mcA^{(n)}\times \mcA^{(n)}$
  we define the \textbf{twist substitution}
  $$(f_1,f_2)\tw (z_1,z_2) = (f_1(z_1)+ f_2(z_2),  (f_1(z_2)+ f_2(z_1)).$$
  
  \begin{lem}\label{twis}
      $((f_1,f_2)\mtw (f_3,f_4))\tw (z_1,z_2) =  (f_1,f_2)\tw ((f_3,f_4)\tw (z_1,z_2))$
  \end{lem}
  \begin{proof}
  By distributivity we may assume that our polynomials are monomials $f_i = \La^{k_i}$ for $1\le i \le 4$.
  Then \begin{equation}\begin{aligned}
((f_1,f_2)\mtw (f_3,f_4))\tw (z_1,z_2) & = (\La ^{k_1k_3+k_2k_4}, (\La ^{k_1k_4+k_2k_3}) (z_1,z_2) \\ & = z_1 ^{k_1k_3+k_2k_4}+ z_2 ^{k_1k_4+k_2k_3},   z_1 ^{k_1k_4+k_2k_3} + z_2 ^{k_1k_3+k_2k_4}  \\ &  = (\La ^{k_1},\La ^{k_2})\tw (  z_1 ^{k_3} + z_2 ^{k_4} , z_1 ^{k_4} + z_2 ^{k_3})
 \\ &  (f_1,f_2)\tw ((f_3,f_4)\tw (z_1,z_2)).
       \end{aligned}  
  \end{equation}
  \end{proof}

  \begin{defn}
  The \textbf{congruence} $\widehat \mcI(S)$ of $S\subseteq \mcA^{(n)}\times \mcA^{(n)}$ then is
  $$(f_1,f_2) \in  \mcA[\La]\times  \mcA[\La]:
  (f_1,f_2)(z_1,z_2) \in \mcA_0 \times \mcA_0, \forall (z_1,z_2)\in S.$$
  
  The congruence of a point is called a \textbf{geometric congruence}.
  \end{defn}

  Lemma~\ref{twis} implies that   $\widehat \mcI(S)$ is a radical congruence, clearly the intersection of a nonempty set of geometric congruences.

  For example for a system with unique negation, if $S = \{ (a,\zero, \dots, \zero),(\zero))\} $
  for $a\in \tT$, then  one has $\widehat \mcI(S) = (\la_1 (-) a+ \mcA_0[\La], a' + \mcA_0[\La])$, a prime congruence.
 
 This opens the door to the Zariski topology on pairs, but pursuing this path is out of the scope of this paper.
 
\section*{Acknowledgement}
\small
The research of the second author is sponsored by Louisiana Board of Regents Targeted Enhancement Grant 090ENH-21. The research of the third author was supported by the ISF grant 1994/20.

\small{The authors would like to thank Sergey Sergeev as well as the referee for helpful suggestions in clarifying the results.}

\end{document}